\documentclass[11pt]{article}
\usepackage{amssymb}
\usepackage{amsmath}
\usepackage{mathrsfs}
\usepackage{graphics}
\usepackage{graphicx}
\usepackage{xcolor}
\usepackage{subfigure}
\usepackage[T1]{fontenc}
\usepackage{latexsym,amssymb,amsmath,amsfonts,amsthm}\usepackage{txfonts}
\newcommand{\be}{\begin{eqnarray}}
\newcommand{\ben}{\begin{eqnarray*}}
\newcommand{\en}{\end{eqnarray}}
\newcommand{\enn}{\end{eqnarray*}}

\newtheorem{theorem}{Theorem}[section]
\newtheorem{lemma}{Lemma}[section]
\newtheorem{prp}[theorem]{Proposition}
\newtheorem{thm}[theorem]{Theorem}

\newtheorem{dfn}{Definition}[section]

\newtheorem{remark}{Remark}

\definecolor{rr}{rgb}{1,0,0}
\begin{document}
\renewcommand{\theequation}{\arabic{section}.\arabic{equation}}
\begin{titlepage}
\title{\bf  Quadratic Transportation Cost Inequality For Scalar Stochastic Conservation Laws
}
\author{ Rangrang Zhang$^{1}$, Tusheng Zhang$^{2}$\\
{\small $^1$  School of Mathematics and Statistics, Beijing Institute of Technology, Beijing 100081, China}\\
{\small $^2$ Department of  Mathematics, University of Manchester, Oxford Road, Manchester M13 9PL, United Kingdom}\\
({\small {\sf rrzhang@amss.ac.cn}, \ {\sf tusheng.zhang@manchester.ac.uk}})}
\date{}
\end{titlepage}
\maketitle

\noindent\textbf{Abstract}:
In this paper, we established a quadratic transportation cost inequality for scalar stochastic conservation laws driven by multiplicative noise. The doubling variables method plays an important role.

\noindent \textbf{AMS Subject Classification}: Primary 60F10; Secondary 60H15, 60G40.

\noindent\textbf{Keywords}: Quadratic transportation cost inequality; conservation laws; Kinetic solutions.

\section{Introduction}
Fix
$T>0$ and let $(\Omega,\mathcal{F},\mathbb{P},\{\mathcal{F}_t\}_{t\in
[0,T]},(\{\beta_k(t)\}_{t\in[0,T]})_{k\in\mathbb{N}})$ be a stochastic basis. Without loss of generality, here the filtration $\{\mathcal{F}_t\}_{t\in [0,T]}$ is assumed to be complete and $\{\beta_k(t)\}_{t\in[0,T]}(k \in\mathbb{N})$ are one-dimensional i.i.d. real-valued  $\{\mathcal{F}_t\}_{t\in [0,T]}-$Wiener processes. The symbol $\mathbb{E}$ denotes the expectation with respect to $\mathbb{P}$.
For fixed $N\in\mathbb{N}$, let $\mathbb{T}^N\subset\mathbb{R}^N$ be the $N-$dimensional torus with the periodic length to be $1$.
Consider the following Cauchy problem for the scalar conservation laws with stochastic forcing
\begin{eqnarray}\label{P-19}
\left\{
  \begin{array}{ll}
  du(t,x)+div A(u(t,x))dt=\sum_{k\geq 1}g_k(x,u(t,x)) d\beta_k(t) \quad {\rm{in}}\ \mathbb{T}^N\times (0,T],\\
u(\cdot,0)=\eta(\cdot) \quad {\rm{on}}\ \mathbb{T}^N.
  \end{array}
\right.
\end{eqnarray}
where $u:(\omega,x,t)\in\Omega\times\mathbb{T}^N\times[0,T]\mapsto u(\omega,x,t):=u(x,t)\in\mathbb{R}$ is a random field, the flux function $A:\mathbb{R}\to\mathbb{R}^N$ and the coefficient $g_k(\cdot,\cdot)$ are measurable and fulfill certain conditions (see Section 2 in below).  Moreover, the initial value $\eta\in L^{\infty}(\mathbb{T}^N)$ is a deterministic function.
\vskip 0.3cm
The purpose of this paper is to establish the quadratic transportation cost inequality for the solution of the stochastic conservation laws. Let us recall the relevant concepts.
Let $(X, d)$ be a metric  space equipped with the  Borel  $\sigma$-field ${\cal B}$. Let $\mu$, $\nu$ be two Borel probability measures on the metric space $(X, d)$. The $L^p$-Wasserstein distance  between $\mu$ and $\nu$ is defined as
$$W_p(\nu, \mu):=\left[\inf \iint_{X\times X}d(x,y)^p\,\pi(\mathrm{d}x,\mathrm{d}y)\right]^{\frac{1}{p}},$$
where the infimum is taken over all probability measures $\pi$ on the product space $X\times X$ with marginals $\mu$ and $\nu$. Recall that the Kullback information (or relative entropy) of $\nu$ with respect to $\mu$ is defined by
\[H(\nu|\mu):=\int_X \log\left(\frac{\mathrm{d}\nu}{\mathrm{d}\mu}\right)\, \mathrm{d}\nu ,\]
if $\nu$ is absolutely continuous with respect to $\mu$, and $+\infty$ otherwise.
\begin{dfn}
We say that a
 measure $\mu$ satisfies the $L^p$-transportation cost inequality if there exists a constant $C>0$ such that for all probability measures $\nu$,
\begin{equation}\label{1.2}
 W_p(\nu, \mu)\leq \sqrt{2C H(\nu| \mu)}.
 \end{equation}
 The case $p=2$ is referred to as the quadratic transportation cost inequality.
\end{dfn}

Transportation cost inequalities have close connections with other functional inequalities, e.g. Poincare inequalities, logarithmic Sobolev inequalities, and they  also imply the  concentration of measure phenomenon.
\vskip 0.3cm
For a measurable subset $A\subset X$ and $r>0$, we denote by $A_r$  the $r$-neighborhood of $A$, namely
$A_r=\{x: d(x,A)<r\}$. We say that $\mu$ has normal concentration (or  Gaussian tail estimates)  on $(X, d)$ if there are constants $C,c>0$ such that for every $r>0$ and every Borel subset $A$ with $\mu(A)\geq \frac{1}{2}$,
\begin{equation}\label{0.1}
1-\mu(A_r)\leq Ce^{-cr^2}.
\end{equation}

The fact that the $L^1$-transportation cost inequality implies normal concentration was obtained in \cite{M-1,M-2} by Marton and in \cite{T1,T2,T3} by Talagrand. An elegant, simple proof of this fact is also contained in the book \cite{L}. The connection of the quadratic transportation cost inequality  with other functional inequalities was studied in \cite{OV} by Otto and Villani (see also \cite{L}). For other related interesting works, we refer the reader to  \cite{GRS}, \cite{LW}, \cite{PS}, \cite{PS1}.

\vskip 0.3cm

In the past decades, many people  established quadratic transportation cost inequalities  for various kinds of interesting measures. Let us  mention several papers which are relevant to our work. The transportation cost inequalities for stochastic differential equations were obtained by H. Djellout, A. Guillin and L. Wu in \cite{DGW}. The measure concentration for multidimensional diffusion processes with reflecting boundary conditions was considered by S. Pal in \cite{P}. The quadratic transportation cost inequalities for  stochastic partial differential equations (SPDEs) driven by Gaussian noise which is white in time and colored in space were obtained by A. S. Ustunel in \cite{U}. In the article \cite{BH2}, the authors obtained the quadratic transportation cost inequality under the $L^2$-distance for stochastic heat equations. In \cite{KS}, the authors established the quadratic transportation cost inequality  for more general stochastic partial differential  equations (SPDEs) under  the $L^2$-distance and under the uniform distance for the case of additive noise. In \cite{SZ}, the authors obtained the quadratic transportation cost inequality  for stochastic heat equations equations driven by multiplicative space-time white noise  under the uniform distance.

\vskip 0.3cm
On the other hand, both deterministic ($g_k=0$) and stochastic conservation laws have been studied extensively  by many people.
Conservation law is fundamental to our understanding of the space-time evolution laws of interesting physical quantities. For more background on this model, we refer the readers to
the monograph \cite{Dafermos}, the work of Ammar,
Wittbold and Carrillo \cite{K-P-J} and references therein. As we know, the Cauchy problem
for the deterministic first-order PDE (\ref{P-19}) does not admit any (global) smooth solutions, but there exist infinitely many weak solutions to the deterministic Cauchy problem. To solve the problem of non-uniqueness, an additional entropy condition was added to identify the physical weak solution. Under this condition,
the notion of entropy solutions for the deterministic first-order scalar conservation laws was introduced by Kru\v{z}kov \cite{K-69,K-70}.
The kinetic formulation of weak entropy solution of the Cauchy problem for a general multi-dimensional scalar conservation laws (also called the kinetic system), was derived by Lions, Perthame and Tadmor in \cite{L-P-T}.
\vskip 0.3cm
In recent years,  the stochastic conservation law has been developed rapidly. We refer the reader to the references \cite{K},  \cite{V-W},\cite{F-N}, \cite{DWZZ} etc. We
  particularly mention the paper \cite{D-V-1} in which the authors  proved the existence and uniqueness of kinetic solution to the Cauchy problem for (\ref{P-19}) in any dimension.
 In addition, the long-time behavior of the first-order scalar conservation laws has been studied in the paper  \cite{D-V-2}.
Recently, combining techniques used in the context of kinetic solutions as well as new results on large deviations, Dong et al. \cite{DWZZ} established Freidlin-Wentzell's type large deviation principles (LDP) for the kinetic solution to the scalar stochastic conservative laws.

\smallskip
 The purpose of this paper is to establish the quadratic transportation cost inequality for the kinetic solution of the scalar stochastic conservation laws, which in particular implies the concentration phenomenon of the law of the solution. To our knowledge, the present paper is the first work towards proving the transportation cost inequality directly for the kinetic solutions to the scalar stochastic conservation laws.
Due to the lack of viscous term, the kinetic solutions of (\ref{P-19}) are living in a rather irregular space $L^1([0, T],L^1(\mathbb{T}^N))$, we will use
 the doubling variables method as in the work \cite{D-V-1}. Differ from \cite{D-V-1}, we need to deal with the martingale term carefully to derive a proper bound, which can ensure the application of Gronwall inequality to get an appropriate norm estimation (see (\ref{qq-10})). As an important part of the proof, we also need to make some higher order estimates of the error term than \cite{D-V-1}, which is nontrivial and completely new.

This paper is organized as follows. In Section 2, we lay out  the precise setup for the stochastic conservation law and recall some  of the known results. Section 3 is devoted to the proof of the transportation cost inequality.

\section{Framework}\label{S:2}
\setcounter{equation}{0}

In this section, we will lay out  the precise setup for the stochastic conservation law and recall some  results which will be used later.

\subsection{Kinetic solution}

We will follow closely the framework of \cite{D-V-1}.
Let $\|\cdot\|_{L^p}$ denote the norm of usual Lebesgue space $L^p(\mathbb{T}^N)$ for $p\in [1,\infty]$. In particular, set $H=L^2(\mathbb{T}^N)$ with the corresponding norm $\|\cdot\|_H$.
  $C_b$ represents the space of bounded, continuous functions and $C^1_b$ stands for the space of bounded, continuously differentiable functions having bounded first order derivative. Define the function $f(x,t,\xi):=I_{u(x,t)>\xi}$, which is the characteristic function of the subgraph of $u$. We write $f:=I_{u>\xi}$ for short.
 Moreover, denote by the brackets $\langle\cdot,\cdot\rangle$ the duality between $C^{\infty}_c(\mathbb{T}^N\times \mathbb{R})$ and the space of distributions over $\mathbb{T}^N\times \mathbb{R}$. In what follows, with a slight abuse of the notation $\langle\cdot,\cdot\rangle$, we denote the following integral by
\[
\langle F, G \rangle:=\int_{\mathbb{T}^N}\int_{\mathbb{R}}F(x,\xi)G(x,\xi)dxd\xi, \quad F\in L^p(\mathbb{T}^N\times \mathbb{R}), G\in L^q(\mathbb{T}^N\times \mathbb{R}),
\]
where $1\leq p\leq +\infty$, $q:=\frac{p}{p-1}$ is the conjugate exponent of $p$. In particular, when $p=1$, we set $q=\infty$ by convention. For a measure $m$ on the Borel measurable space $\mathbb{T}^N\times[0,T]\times \mathbb{R}$, the shorthand $m(\phi)$ is defined by
\[
m(\phi):=\langle m, \phi \rangle([0,T]):=\int_{\mathbb{T}^N\times[0,T]\times \mathbb{R}}\phi(x,t,\xi)dm(x,t,\xi), \quad  \phi\in C_b(\mathbb{T}^N\times[0,T]\times \mathbb{R}).
\]
In the sequel,  the notation $a\lesssim b$ for $a,b\in \mathbb{R}$  means that $a\leq \mathcal{D}b$ for some constant $\mathcal{D}> 0$ independent of any parameters.

\subsection{Hypotheses}
For the flux function $A$ and the coefficients  of (\ref{P-19}), we assume that
\begin{description}
  \item[\textbf{Hypothesis H}] The flux function $A$ belongs to $C^2(\mathbb{R};\mathbb{R}^N)$ and its derivative $a:=A'$ is of polynomial growth with degree $q_0>1$. That is, there exists a constant $C(q_0)\geq0$ such that
     \begin{eqnarray}\label{qeq-22}
    |a(\xi)|\leq C(q_0)(1+|\xi|^{q_0}),\quad |a(\xi)-a(\zeta)|\leq \Upsilon(\xi,\zeta)|\xi-\zeta|,
      \end{eqnarray}
      where $\Upsilon(\xi,\zeta):=C(q_0)(1+|\xi|^{q_0-1}+|\zeta|^{q_0-1})$.

     Moreover, we assume that $g_k\in C(\mathbb{T}^N\times \mathbb{R})$ satisfies the following bounds
\begin{eqnarray}\label{e-5}
|g_k(x,u)|&\leq& C^0_k, \quad \sum_{k\geq 1}|C^0_k|^2\leq D_0,\\
\label{e-6}
|g_k(x,u)-g_k(y,v)|&\leq& C^1_k(|x-y|+|u-v|),\quad \sum_{k\geq 1}|C^1_{k}|^2\leq \frac{D_1}{2},
\end{eqnarray}

for $x, y\in \mathbb{T}^N, u,v\in \mathbb{R}$, where $C^0_k, C^1_k, D_0, D_1$ are positive constants.
\end{description}
The hypothesis H implies  that
\begin{eqnarray}\label{equ-28}
G^2(x,u):=\sum_{k\geq 1}|g_k(x,u)|^2&\leq& D_0,\\
\label{equ-29}
\sum_{k\geq 1}|g_k(x,u)-g_k(y,v)|^2&\leq& D_1\Big(|x-y|^2+|u-v|^2\Big).
\end{eqnarray}

\subsection{Kinetic solution}
Let us recall the notion of a kinetic solution to equation (\ref{P-19}) from \cite{D-V-1}.
\begin{dfn}(Kinetic measure)\label{dfn-3}
 A map $m$ from $\Omega$ to the set of non-negative, finite measures over $\mathbb{T}^N\times [0,T]\times\mathbb{R}$ is said to be a kinetic measure, if
\begin{description}
  \item[1.] $ m $ is measurable, that is, for each $\phi\in C_b(\mathbb{T}^N\times [0,T]\times \mathbb{R}), \langle m, \phi \rangle: \Omega\rightarrow \mathbb{R}$ is measurable,
  \item[2.] $m$ vanishes for large $\xi$, i.e.,
\begin{eqnarray}\label{equ-37}
\lim_{R\rightarrow +\infty}\mathbb{E}[m(\mathbb{T}^N\times [0,T]\times B^c_R)]=0,
\end{eqnarray}
where $B^c_R:=\{\xi\in \mathbb{R}, |\xi|\geq R\}$,
  \item[3.] for every $\phi\in C_b(\mathbb{T}^N\times \mathbb{R})$, the process
\[
(\omega,t)\in\Omega\times[0,T]\mapsto \langle m,\phi\rangle([0,t]):= \int_{\mathbb{T}^N\times [0,t]\times \mathbb{R}}\phi(x,\xi)dm(x,s,\xi)\in\mathbb{R}
\]
is predictable.
\end{description}
\end{dfn}
\begin{remark}\label{r-1}
  For any $\phi\in C_b(\mathbb{T}^N\times \mathbb{R})$ and kinetic measure $m$, define $A_t:=\langle m, \phi\rangle([0,t]),$
   then a.s., $t\mapsto A_t$ is a right continuous function of finite variation. Moreover, the function $A$ has left limits at any point $t\in (0,T]$. We write $A_{t^-}=\lim_{s\uparrow t}A_s$ and set $A_{0^-}=0$.
  As a result, $A_{t^-}=\langle m, \phi\rangle([0,t))$, which is c\`{a}gl\`{a}d (left continuous with right limits).

\end{remark}

\begin{dfn}(Kinetic solution)\label{dfn-1}
Let $\eta\in L^{\infty}(\mathbb{T}^N)$. A measurable function $u: \mathbb{T}^N\times [0,T]\times\Omega\rightarrow \mathbb{R}$ is called a kinetic solution to (\ref{P-19}) with initial datum $\eta$, if
\begin{description}
  \item[1.] $(u(t))_{t\in[0,T]}$ is predictable,
  \item[2.] for any $p\geq1$, there exists $C_p\geq0$ such that
\begin{eqnarray}\label{qq-8}
\mathbb{E}\left(\underset{0\leq t\leq T}{{\rm{ess\sup}}}\ \|u(t)\|^p_{L^p(\mathbb{T}^N)}\right)\leq C_p,
\end{eqnarray}
\item[3.] there exists a kinetic measure $m$ such that $f:= I_{u>\xi}$ satisfies: for all $\varphi\in C^1_c(\mathbb{T}^N\times [0,T)\times \mathbb{R})$,
\begin{eqnarray}\label{q-2.1}
&&\int^T_0\langle f(t), \partial_t \varphi(t)\rangle dt+\langle f_0, \varphi(0)\rangle +\int^T_0\langle f(t), a(\xi)\cdot \nabla \varphi (t)\rangle dt\\ \notag
&=& -\sum_{k\geq 1}\int^T_0\int_{\mathbb{T}^N} g_k(x, u(t,x))\varphi (x,t,u(x,t))dxd\beta_k(t) \\
\notag
&& -\frac{1}{2}\int^T_0\int_{\mathbb{T}^N}\partial_{\xi}\varphi (x,t,u(x,t))G^2(x,u(t,x))dxdt+ m(\partial_{\xi} \varphi), \ a.s. ,
\end{eqnarray}
where $f_0=I_{\eta>\xi}$, $u(t)=u(\cdot,t,\cdot)$ and $G^2=\sum^{\infty}_{k=1}|g_k|^2$.
\end{description}
\end{dfn}

Let $(X,\lambda)$ be a finite measure space. For some measurable function $u: X\rightarrow \mathbb{R}$, define $f: X\times \mathbb{R}\rightarrow [0,1]$ by $f(z,\xi)=I_{u(z)>\xi}$ a.e.
we use $\bar{f}:=1-f$ to denote its conjugate function. Define $\Lambda_f(z,\xi):=f(z,\xi)-I_{0>\xi}$, which can be viewed as a correction to $f$. Note that $\Lambda_f$ is integrable on $X\times \mathbb{R}$ if $u$ is.
\vskip 0.3cm

It is shown in \cite{D-V-1} that almost surely, for each kinetic solution $u$, the function $f=I_{u(x,t)>\xi}$ admits left and right weak limits at any point $t\in[0,T]$, and the weak form (\ref{q-2.1}) satisfied by a kinetic solution can be strengthened to be weak only respect to $x$ and $\xi$. More precisely,  the following results are obtained.
\begin{prp}(\cite{D-V-1}, Left and right weak limits)\label{prp-3} Let $u$ be a kinetic solution to (\ref{P-19}) with initial value $\eta$. Then $f=I_{u(x,t)>\xi}$  admits, almost surely, left and right limits respectively at every point $t\in [0,T]$. More precisely, for any  $t\in [0,T]$, there exist kinetic functions $f^{t\pm}$ on $\Omega\times \mathbb{T}^N\times \mathbb{R}$ such that $\mathbb{P}-$a.s.
\begin{eqnarray}\label{e-50}
\langle f(t-r),\varphi\rangle\rightarrow \langle f^{t-},\varphi\rangle
\end{eqnarray}
and
\begin{eqnarray}\label{e-51}
\langle f(t+r),\varphi\rangle\rightarrow \langle f^{t+},\varphi\rangle
\end{eqnarray}
as $r\rightarrow 0$ for all $\varphi\in C^1_c(\mathbb{T}^N\times \mathbb{R})$. Moreover, almost surely,
\[
\langle f^{t+}-f^{t-}, \varphi\rangle=-\int_{\mathbb{T}^N\times[0,T]\times \mathbb{R}}\partial_{\xi}\varphi(x,\xi)I_{\{t\}}(s)dm(x,s,\xi).
\]
In particular, almost surely, the set of $t\in [0,T]$ fulfilling $f^{t+}\neq f^{t-}$ is countable.
\end{prp}
For the function $f=I_{u(x,t)>\xi}$ in Proposition \ref{prp-3}, define $f^{\pm}$ by $f^{\pm}(t)=f^{t \pm}$, $t\in [0,T]$. Since we are dealing with the filtration associated to Brownian motion, both $f^{+}$ and $f^{-}$ are  clearly predictable as well. Also $f=f^+=f^-$ almost everywhere in time and we can take any of them in an integral with respect to the Lebesgue measure or in a stochastic integral. However, if the integral is with respect to a measure, typically a kinetic measure in this article, the integral is not well-defined for $f$ and may differ if one chooses $f^+ $ or $f^-$.

The following result was proved in  \cite{D-V-1}.
 \begin{lemma}\label{lem-1}
 The weak form (\ref{q-2.1}) satisfied by $f= I_{u>\xi}$ can be strengthened to be weak only respect to $x$ and $\xi$. Concretely, for all $t\in [0,T)$ and $\varphi\in C^1_c(\mathbb{T}^N\times \mathbb{R})$, $f= I_{u>\xi}$ satisfies
 \begin{eqnarray}\notag
\langle f^+(t),\varphi\rangle&=&\langle f_{0}, \varphi\rangle+\int^t_0\langle f(s), a(\xi)\cdot \nabla \varphi\rangle ds\\
\notag
&&+\sum_{k\geq 1}\int^t_0\int_{\mathbb{T}^N}\int_{\mathbb{R}}g_k(x,\xi)\varphi(x,\xi)d\nu_{x,s}(\xi)dxd\beta_k(s)\\
\label{qq-17}
&& +\frac{1}{2}\int^t_0\int_{\mathbb{T}^N}\int_{\mathbb{R}}\partial_{\xi}\varphi(x,\xi)G^2(x,\xi)d\nu_{x,s}(\xi)dxds- \langle m,\partial_{\xi} \varphi\rangle([0,t]), \quad a.s.,
\end{eqnarray}
 and we set $f^+(T)=f(T)$.
 \end{lemma}
 Where $\nu_{x,s}(\xi)=-\partial_{\xi}f(x,s,\xi)=\delta_{u(x,s)=\xi}$.

\begin{remark}  By making  modification of the proof of Lemma \ref{lem-1}, we have for all $t\in (0,T]$ and $\varphi\in C^1_c(\mathbb{T}^N\times \mathbb{R})$, $f= I_{u>\xi}$ satisfies
   \begin{eqnarray}\notag
\langle f^-(t),\varphi\rangle&=&\langle f_{0}, \varphi\rangle+\int^t_0\langle f(s), a(\xi)\cdot \nabla \varphi\rangle ds\\
\notag
&&+\sum_{k\geq 1}\int^t_0\int_{\mathbb{T}^N}\int_{\mathbb{R}}g_k(x,\xi)\varphi(x,\xi)d\nu_{x,s}(\xi)dxd\beta_k(s)\\
\label{e-80}
&& +\frac{1}{2}\int^t_0\int_{\mathbb{T}^N}\int_{\mathbb{R}}\partial_{\xi}\varphi(x,\xi)G^2(x,\xi)d\nu_{x,s}(\xi)dxds- \langle m,\partial_{\xi} \varphi\rangle([0,t)), \quad a.s.,
\end{eqnarray}
and we set $ f^-(0)=f_0$.
\end{remark}

The following  well-posedness of (\ref{P-19}) was established  in \cite{D-V-1}.
\begin{thm}\label{thm-4}
(\cite{D-V-1}, Existence and Uniqueness) Let $\eta\in L^{\infty}(\mathbb{T}^N)$. Assume Hypothesis H holds, then there is a unique kinetic solution $u$ to equation (\ref{P-19}) with initial datum $\eta$.
\end{thm}

\section{Transportation cost inequality}
%

Let $\mu$ be the law of the random field solution $u(\cdot, \cdot)$ of SPDE (\ref{P-19}), viewed as a probability measure on $L^1([0, T],L^1(\mathbb{T}^N))$. First we state a lemma which is essentially proved in \cite{KS} describing the probability measures $\nu$ that are absolutely continuous with respect to $\mu$.
\vskip 0.4cm
Let $\nu\ll \mu$ on $L^1([0, T],L^1(\mathbb{T}^N))$.
Define a new probability measure $\mathbb{Q}$ on the filtered probability space $(\Omega, {\cal F}, \{{\cal F}_{t}\}_{0\leq t\leq T}, \mathbb{P})$ by
\begin{align}\label{add 0303.1}
\mathrm{d}\mathbb{Q}:=\frac{\mathrm{d}\nu}{\mathrm{d}\mu}(u) \,\mathrm{d}\mathbb{P} .
\end{align}
Denote the Radon-Nikodym derivative restricted on ${\cal F}_t$ by
\[M_t:=\left. \frac{\mathrm{d}\mathbb{Q}}{\mathrm{d}\mathbb{P}}\right |_{{\cal F}_t}, \quad t\in [0, T].\]
Then $M_t, t\in [0, T]$ forms a $\mathbb{P}$-martingale. A variant of the  following result was proved in \cite{KS}.
\begin{lemma}
There exists an adapted stochastic process $h=\{h(s)=(h_1(s),h_2(s),\cdot\cdot\cdot,)\in l^2, s\geq 0\}$ such that $\mathbb{Q}-a.s.$ for all $t\in [0, T]$,
\begin{align*}
\int_0^t |h|_{l^2}^2(s)\,\mathrm{d}s<\infty
\end{align*}
and $\widetilde{\beta}_k: [0, T]\rightarrow \mathbb{R}$ defined by
\begin{align}\label{4.2}
\widetilde{\beta}_k(t):=\beta_k(t)-\int_0^t h_k(s)\,\mathrm{d}s,
\end{align}
are independent  Brownian motions under the measure $\mathbb{Q}$. Moreover,
\begin{align}\label{4.3}
M_t=\exp\left (\sum_{k=1}^{\infty}\int_0^th_k(s)\,\mathrm{d}\beta_k(s)-\frac{1}{2}\int_0^t |h|_{l^2}^2(s)\,\mathrm{d}s\right ), \quad \mathbb{Q}-a.s.,
\end{align}
and
\begin{align}\label{4.4}
H(\nu|\mu)=\frac{1}{2}\mathbb{E}^{\mathbb{Q}}\left[\int_0^T|h|_{l^2}^2(s)\,\mathrm{d}s\right],
\end{align}
where $\mathbb{E}^{\mathbb{Q}}$ stands for the expectation under the measure $\mathbb{Q}$.
\end{lemma}
\begin{thm}\label{thm-5}
Let $\eta\in L^{\infty}(\mathbb{T}^N)$. Assume Hypothesis H holds. Then the law $\mu$ of the solution of the stochastic conservation law (\ref{P-19}) satisfies the quadratic transportation cost inequality on the space $L^1([0,T],L^1(\mathbb{T}^N))$.
\end{thm}
\noindent {\bf Proof}.
Take $\nu\ll \mu$ on $L^1([0, T],L^1(\mathbb{T}^N))$.
Let  $\mathbb{Q}$ be the probability measure defined as  in(\ref{add 0303.1}).
Let $h(t)$ be the corresponding stochastic process appeared in Lemma 3.1. Then, by the Girsanov theorem  the solution $u(t)$ of equation (\ref{P-19}) satisfies the following stochastic partial differential equation (SPDE) under the measure $\mathbb{Q}$,
\begin{eqnarray}\label{q-4.1}
\left\{
  \begin{array}{ll}
  du^h(t,x)+div A(u^h(t,x))dt=\sum_{k\geq 1}g_k(x,u^h(t,x)) d\widetilde{\beta}_k(t)+\sum_{k\geq 1}g_k(x,u^h(t,x))h_k(t)dt \quad {\rm{in}}\ \mathbb{T}^N\times (0,T],\\
u^h(\cdot,0)=\eta(\cdot) \quad {\rm{on}} \ \mathbb{T}^N.
  \end{array}
\right.
\end{eqnarray}
Similar to Lemma 2.1, we can show that the  kinetic solution $u^h$ satisfies that  for any $p\geq1$, there exists $C_p\geq0$ such that
\begin{eqnarray}\label{equation-1}
\mathbb{E}^{\mathbb{Q}}\left(\underset{0\leq t\leq T}{{\rm{ess\sup}}}\ \|u^h(t)\|^p_{L^p(\mathbb{T}^N)}\right)\leq C_p,
\end{eqnarray}
and there exists a kinetic measure $m^h$ such that
for all $t\in [0,T)$ and $\varphi\in C^1_c(\mathbb{T}^N\times \mathbb{R})$, $f:= I_{u^h>\xi}$ satisfies
 \begin{eqnarray}\notag
\langle f^+(t),\varphi\rangle&=&\langle f_{0}, \varphi\rangle+\int^t_0\langle f(s), a(\xi)\cdot \nabla \varphi\rangle ds\\
\notag
&& +\sum_{k\geq 1}\int^t_0\int_{\mathbb{T}^N}\int_{\mathbb{R}}g_k(x,\xi)\varphi(x,\xi)d\nu^h_{x,s}(\xi)dx d\widetilde{\beta}_k(s)\\ \notag
&& +\sum_{k\geq 1}\int^t_0\int_{\mathbb{T}^N}\int_{\mathbb{R}} g_k(x, \xi)\varphi (x,\xi)h_k(s)d\nu^h_{x,s}(\xi)dxds \\
\label{equation-2}
&& +\frac{1}{2}\int^t_0\int_{\mathbb{T}^N}\int_{\mathbb{R}}\partial_{\xi}\varphi(x,\xi)G^2(x,\xi)d\nu^h_{x,s}(\xi)dxds- \langle m^h,\partial_{\xi} \varphi\rangle([0,t]), \quad a.s.,
\end{eqnarray}
where $\nu^{h}_{x,s}(\xi)=-\partial_{\xi}f(x,s,\xi)=\delta_{u^{h}(x,s)=\xi}$ and we set $f^+(T)=f(T)$.

\vskip 0.4cm

Consider the solution of the following SPDE:
\begin{eqnarray}\label{q-4.2}
\left\{
  \begin{array}{ll}
  du(t,x)+div A(u(t,x))dt=\sum_{k\geq 1}g_k(x,u(t,x)) d\widetilde{\beta}_k(t) \quad {\rm{in}}\ \mathbb{T}^N\times (0,T],\\
u(\cdot,0)=\eta(\cdot) \quad {\rm{on}}\ \mathbb{T}^N.
  \end{array}
\right.
\end{eqnarray}
By Lemma 3.1, it follows that under the measure $\mathbb{Q}$, the law of $(u,u^h)$ forms a coupling of $(\mu, \nu)$. Therefore by the definition of the Wasserstein distance,
\[
W_2(\nu, \mu)^2\leq \mathbb{E}^{\mathbb{Q}}\left[\left |\int_0^T\int_{\mathbb{T}^N}|u(t,x)-u^h(t,x)|dtdx\right|^2 \right].
\]
In view of (\ref{4.4}), to prove the quadratic transportation cost inequality
\begin{align}
   W_2(\nu, \mu)\leq \sqrt{2C H(\nu|\mu)} ,
\end{align}
it is sufficient to show that
\begin{align}\label{111.1}
\mathbb{E}^{\mathbb{Q}}\left[\left |\int_0^T\int_{\mathbb{T}^N}|u(t,x)-u^h(t,x)|dtdx\right|^2 \right]
\leq C \mathbb{E}^{\mathbb{Q}}\left[\int_0^T|h|_{l^2}^2(s)\,\mathrm{d}s\right]
\end{align}
when the right hand side of (\ref{111.1}) is finite.
\vskip 0.3cm
For simplicity, in the sequel we still denote $\mathbb{E}^{\mathbb{Q}}$ by the symbol $\mathbb{E}$ and denote $\widetilde{\beta}_k$ by ${\beta}_k$.
The proof of (\ref{111.1}) is technical and lengthy. It is divided into the following two propositions.
\vskip 0.3cm
Following the idea of the proof Proposition 13 in \cite{D-V-1} and using the doubling variables method, we have the following result relating the two kinetic solution $u$ and $u^{h}$. As the proof is very similar to that of Proposition 13 in \cite{D-V-1}, we omit the proof and refer the reader to  \cite{D-V-1}.
\begin{prp}\label{prp-1}
Assume Hypothesis H is in place. Let $u$ and $u^{h}$ be the kinetic solution of (\ref{q-4.2}) and (\ref{q-4.1}) , respectively. Then, for all $0< t< T$, and non-negative test functions $\rho\in C^{\infty}(\mathbb{T}^N), \psi\in C^{\infty}_c(\mathbb{R})$, the corresponding functions $f_1(x,t,\xi):=I_{u^{h}(x,t)>\xi}$ and $f_2(y,t,\zeta):=I_{u(y,t)>\zeta}$ satisfy the following
\begin{eqnarray}\notag
&&\int_{(\mathbb{T}^N)^2}\int_{\mathbb{R}^2}\rho (x-y)\psi(\xi-\zeta)(f^{\pm}_1(x,t,\xi)\bar{f}^{\pm}_2(y,t,\zeta)+\bar{f}^{\pm}_1(x,t,\xi)f^{\pm}_2(y,t,\zeta))d\xi d\zeta dxdy\\
\notag
&\leq & \int_{(\mathbb{T}^N)^2}\int_{\mathbb{R}^2}\rho (x-y)\psi(\xi-\zeta)(f_{1,0}(x,\xi)\bar{f}_{2,0}(y,\zeta)+\bar{f}_{1,0}(x,\xi)f_{2,0}(y,\zeta))d\xi d\zeta dxdy\\
\label{eq-14-1}
&& +I(t)+J(t)+K(t)+H(t), \quad a.s.,
\end{eqnarray}
where
\begin{eqnarray*}
I(t)&=&\int^t_0\int_{(\mathbb{T}^N)^2}\int_{\mathbb{R}^2}(f_1\bar{f}_2+\bar{f}_1f_2)(a (\xi)-a(\zeta))\cdot\nabla_x\alpha d\xi d\zeta dxdyds,\\
J(t)&=&\int^t_0\int_{(\mathbb{T}^N)^2}\int_{\mathbb{R}^2}\alpha \sum_{k\geq 1}|g_k(x,\xi)-g_k(y,\zeta)|^2d\nu^{1}_{x,s}\otimes \nu^{2}_{y,s}(\xi,\zeta)dxdyds,\\
K(t)&=& 2\sum_{k\geq 1}\int^t_0\int_{(\mathbb{T}^N)^2}\int_{\mathbb{R}^2}(g_k(x,\xi)-g_k(y,\zeta))\rho(x-y)\chi(\xi,\zeta)d\nu^{1}_{x,s}\otimes \nu^{2}_{y,s}(\xi,\zeta)dxdyd\beta_k(s),\\
H(t)&=&2\sum_{k\geq 1}\int^t_0\int_{(\mathbb{T}^N)^2}\int_{\mathbb{R}^2}g_k(x,\xi)h_k(s)\rho(x-y)\chi(\xi,\zeta)d\nu^{1}_{x,s}\otimes \nu^{2}_{y,s}(\xi,\zeta)dxdyds,
\end{eqnarray*}
with $f_{1,0}(x,\xi)=I_{\eta(x)>\xi}, f_{2,0}(y,\zeta)=I_{\eta(y)>\zeta}$, $\alpha=\rho (x-y)\psi(\xi-\zeta)$, $\nu^{1}_{x,s}=-\partial_{\xi}f_1(s,x,\xi)=\delta_{u^{h}(x,s)=\xi}, \nu^{2}_{y,s}=\partial_{\zeta}\bar{f}_2(s,y,\zeta)=\delta_{u(y,s)=\zeta}$ and $\chi(\xi,\zeta)=\int^{\xi}_{-\infty}\psi(\xi'-\zeta)d\xi'=\int^{\xi-\zeta}_{-\infty}\psi(y)dy$.
\end{prp}

The statement (\ref{111.1}) is contained in the next proposition.
\begin{prp}\label{prp-2}
 For $T>0$, it holds that
\begin{eqnarray}
\mathbb{E}\Big|\int_0^T\|u^{h}(t)-u(t)\|_{L^1(\mathbb{T}^N)}dt\Big|^2\leq C\mathbb{E}\Big[\int^T_0|h|^2_{l^2}(t)dt\Big],
\end{eqnarray}
where $C=C(T,D_0,D_1)$.
\end{prp}
\begin{proof}
Let $\rho_{\gamma}, \psi_{\delta}$ be approximations to the identity on $\mathbb{T}^N$ and $\mathbb{R}$, respectively. That is, let $\rho\in C^{\infty}(\mathbb{T}^N)$, $\psi\in C^{\infty}_c(\mathbb{R})$ be symmetric non-negative functions such as $\int_{\mathbb{T}^N}\rho =1$, $\int_{\mathbb{R}}\psi =1$ and supp$\psi \subset (-1,1)$. We define
\[
\rho_{\gamma}(x)=\frac{1}{\gamma^N}\rho\Big(\frac{x}{\gamma}\Big), \quad \psi_{\delta}(\xi)=\frac{1}{\delta}\psi\Big(\frac{\xi}{\delta}\Big).
\]
Letting $\rho:=\rho_{\gamma}(x-y)$ and $\psi:=\psi_{\delta}(\xi-\zeta)$ in Proposition \ref{prp-1}, we get from (\ref{eq-14-1}) that
\begin{eqnarray}\label{111.1-1}
&&\int_{(\mathbb{T}^N)^2}\int_{\mathbb{R}^2}\rho_{\gamma} (x-y)\psi_{\delta}(\xi-\zeta)(f^{\pm}_1(x,t,\xi)\bar{f}^{\pm}_2(y,t,\zeta)+\bar{f}^{\pm}_1(x,t,\xi)f^{\pm}_2(y,t,\zeta))d\xi d\zeta dxdy \nonumber\\
&\leq & \int_{(\mathbb{T}^N)^2}\int_{\mathbb{R}^2}\rho_{\gamma} (x-y)\psi_{\delta}(\xi-\zeta)(f_{1,0}(x,\xi)\bar{f}_{2,0}(y,\zeta)+\bar{f}_{1,0}(x,\xi)f_{2,0}(y,\zeta))d\xi d\zeta dxdy\nonumber\\
&&\  +\tilde{I}(t)+\tilde{J}(t)+\tilde{K}(t)+\tilde{H}(t),\quad a.s.,
\end{eqnarray}
where $\tilde{I}, \tilde{J}, \tilde{K}, \tilde{H}$ are the corresponding terms $I,J,K,H$ in the statement of Proposition \ref{prp-1} with $\rho$, $\psi$ replaced by $\rho_{\gamma}$, $\psi_{\delta}$, respectively. For simplicity, we still use the notation:
 \[
 \chi(\xi,\zeta)=\int_{-\infty}^{\xi-\zeta}\psi_{\delta}(y)dy.
 \]
With an eye on the following identity,
\begin{eqnarray}\label{111.1-2}
\int_{\mathbb{R}}I_{u^{h,\pm}>\xi}\overline{I_{u^{\pm}>\xi}}d\xi=(u^{h,\pm}-u^{\pm})^+,
\quad
\int_{\mathbb{R}}\overline{I_{u^{h,\pm}>\xi}}I_{u^{\pm}>\xi}d\xi=(u^{h,\pm}-u^{\pm})^-,
\end{eqnarray}
we will start with the estimate (\ref{111.1-1}) and eventually let $\gamma$, $\delta$ appropriately tend to zero to prove the proposition.
For any $t\in [0,T]$, define the error term
\begin{eqnarray}\notag
&&\mathcal{E}_t(\gamma,\delta)\\ \notag
&:=&\int_{(\mathbb{T}^N)^2}\int_{\mathbb{R}^2}(f^{\pm}_1(x,t,\xi)\bar{f}^{\pm}_2(y,t,\zeta)+\bar{f}^{\pm}_1(x,t,\xi){f}^{\pm}_2(y,t,\zeta))\rho_{\gamma}(x-y)\psi_{\delta}(\xi-\zeta)dxdyd\xi d\zeta\\
\label{qq-3}
&&-\int_{\mathbb{T}^N}\int_{\mathbb{R}}(f^{\pm}_1(x,t,\xi)\bar{f}^{\pm}_2(x,t,\xi)+\bar{f}^{\pm}_1(x,t,\xi)f^{\pm}_2(x,t,\xi))d\xi dx.
\end{eqnarray}
Using $\int_{\mathbb{R}}\psi_{\delta}(\xi-\zeta)d\zeta=1$, $\int^{\xi}_{\xi-\delta}\psi_{\delta}(\xi-\zeta)d\zeta=\frac{1}{2}$ and $\int_{(\mathbb{T}^N)^2}\rho_{\gamma}(x-y)dxdy\leq1$, we find that
\begin{eqnarray}\notag
&&\Big|\int_{(\mathbb{T}^N)^2}\int_{\mathbb{R}}\rho_{\gamma}(x-y)f^{\pm}_1(x,t,\xi)\bar{f}^{\pm}_2(y,t,\xi)d\xi dxdy\\ \notag
&&-\int_{(\mathbb{T}^N)^2}\int_{\mathbb{R}^2}f^{\pm}_1(x,t,\xi)\bar{f}^{\pm}_2(y,t,\zeta)\rho_{\gamma}(x-y)\psi_{\delta}(\xi-\zeta)dxdyd\xi d\zeta\Big|\\ \notag
&=&\Big|\int_{(\mathbb{T}^N)^2}\rho_{\gamma}(x-y)\int_{\mathbb{R}}I_{u^{h, \pm}(x,t)>\xi}\int_{\mathbb{R}}\psi_{\delta}(\xi-\zeta)(I_{u^{\pm}(y,t)\leq\xi}-I_{u^{\pm}(y,t)\leq \zeta})d\zeta d\xi dxdy\Big|\\ \notag
&\leq&\int_{(\mathbb{T}^N)^2}\int_{\mathbb{R}}\rho_{\gamma}(x-y)I_{u^{h,\pm}(x,t)>\xi}\int^{\xi}_{\xi-\delta}\psi_{\delta}(\xi-\zeta)I_{\zeta<u^{\pm}(y,t)\leq\xi} d\zeta d\xi dxdy\\ \notag
&&\ +\int_{(\mathbb{T}^N)^2}\int_{\mathbb{R}}\rho_{\gamma}(x-y)I_{u^{h,\pm}(x,t)>\xi}\int^{\xi+\delta}_{\xi}\psi_{\delta}(\xi-\zeta)I_{\xi<u^{\pm}(y,t)\leq\zeta} d\zeta d\xi dxdy\\ \notag
&\leq& \frac{1}{2}\int_{(\mathbb{T}^N)^2}\rho_{\gamma}(x-y)I_{{u^{h,\pm}(x,t)>u^{\pm}(y,t)}}\int^{min\{u^{h,\pm}(x,t),u^{\pm}(y,t)+\delta\}}_{u^{\pm}(y,t)}d\xi dxdy\\ \notag
&&\ +\frac{1}{2}\int_{(\mathbb{T}^N)^2}\rho_{\gamma}(x-y)I_{{u^{\pm}(y,t)-\delta<u^{h,\pm}(x,t)}}\int^{min\{u^{h,\pm}(x,t),u^{\pm}(y,t)\}}_{u^{\pm}(y,t)-\delta}d\xi dxdy\\ \notag
&=& \frac{\delta}{2}\int_{(\mathbb{T}^N)^2}\rho_{\gamma}(x-y) I_{{ u^{h,\pm}(x,t)>u^{\pm}(y,t)+\delta }}dxdy\\ \notag
&&+\frac{1}{2}\int_{(\mathbb{T}^N)^2}\rho_{\gamma}(x-y)I_{{u^{\pm}(y,t)< u^{h,\pm}(x,t)\leq u^{\pm}(y,t)+\delta }}(u^{h,\pm}(x,t)-u^{\pm}(y,t) dxdy\\ \notag
&&+\frac{\delta}{2}\int_{(\mathbb{T}^N)^2}\rho_{\gamma}(x-y)I_{{u^{\pm}(y,t)<u^{h,\pm}(x,t)}}dxdy\\ \notag
&&+\frac{1}{2}\int_{(\mathbb{T}^N)^2}\rho_{\gamma}(x-y)I_{{u^{\pm}(y,t)-\delta<u^{h,\pm}(x,t)\leq u^{\pm}(y,t)}}(u^{h,\pm}(x,t)-u^{\pm}(y,t)+\delta) dxdy\\
\label{e-23}
&\leq & 2\delta, \quad a.s..
\end{eqnarray}
Similarly, we have
\begin{eqnarray}\notag
&&\Big|\int_{(\mathbb{T}^N)^2}\int_{\mathbb{R}}\rho_{\gamma}(x-y)\bar{f}^{\pm}_1(x,t,\xi){f}^{\pm}_2(y,t,\xi)d\xi dxdy\\
\label{e-22}
&&-\int_{(\mathbb{T}^N)^2}\int_{\mathbb{R}^2}\bar{f}^{\pm}_1(x,t,\xi){f}^{\pm}_2(y,t,\zeta)\rho_{\gamma}(x-y)\psi_{\delta}(\xi-\zeta)dxdyd\xi d\zeta\Big|
\leq  2\delta, \quad a.s..
\end{eqnarray}
Moreover, when $\gamma$ is small enough, it follows that
\begin{eqnarray}\notag
&&\Big|\int_{(\mathbb{T}^N)^2}\int_{\mathbb{R}}\rho_{\gamma}(x-y)f^{\pm}_1(x,t,\xi)\bar{f}^{\pm}_2(y,t,\xi)d\xi dydx-\int_{\mathbb{T}^N}\int_{\mathbb{R}}f^{\pm}_1(x,t,\xi)\bar{f}^{\pm}_2(x,t,\xi)d\xi dx\Big|\\ \notag
&=&\Big|\int_{(\mathbb{T}^N)^2}\int_{\mathbb{R}}\rho_{\gamma}(x-y)f^{\pm}_1(x,t,\xi)\bar{f}^{\pm}_2(y,t,\xi)d\xi dydx-\int_{\mathbb{T}^N}\int_{|z|<\gamma}\int_{\mathbb{R}}\rho_{\gamma}(z)f^{\pm}_1(x,t,\xi)\bar{f}^{\pm}_2(x,t,\xi)d\xi dzdx\Big|\\ \notag
&=&\Big|\int_{(\mathbb{T}^N)^2}\int_{\mathbb{R}}\rho_{\gamma}(x-y)f^{\pm}_1(x,t,\xi)(\bar{f}^{\pm}_2(y,t,\xi)-\bar{f}^{\pm}_2(x,t,\xi))d\xi dydx\Big|\\ \notag
&\leq&\sup_{|z|<\gamma}\int_{\mathbb{T}^N}\int_{\mathbb{R}}f^{\pm}_1(x,t,\xi)|\bar{f}^{\pm}_2(x-z,t,\xi)-\bar{f}^{\pm}_2(x,t,\xi)|d\xi dx\\ \label{e-43}
&\leq& 
 \sup_{|z|<\gamma}\int_{\mathbb{T}^N}\int_{\mathbb{R}}|\Lambda_{f^{\pm}_2}(x-z,t,\xi)-\Lambda_{f^{\pm}_2}(x,t,\xi)|d\xi dx.
\end{eqnarray}
As  $\Lambda_{f_2}$ is integrable, we have for a countable sequence $\gamma_n\downarrow 0$, (\ref{e-43}) holds a.s. for all $n$, hence, passing to the limit $n\rightarrow \infty$, we get
\begin{eqnarray}\label{qq-1}
\lim_{n\rightarrow \infty}\Big|\int_{(\mathbb{T}^N)^2}\int_{\mathbb{R}}\rho_{\gamma_n}(x-y)f^{\pm}_1(x,t,\xi)\bar{f}^{\pm}_2(y,t,\xi)d\xi dxdy-\int_{\mathbb{T}^N}\int_{\mathbb{R}}f^{\pm}_1(x,t,\xi)\bar{f}^{\pm}_2(x,t,\xi)d\xi dx\Big|= 0, \ a.s..
\end{eqnarray}
Similarly, it holds that
\begin{eqnarray}\label{qq-2}
\lim_{n\rightarrow \infty}\Big|\int_{(\mathbb{T}^N)^2}\int_{\mathbb{R}}\rho_{\gamma_n}(x-y)\bar{f}^{\pm}_1(x,t,\xi)f^{\pm}_2(y,t,\xi)d\xi dxdy-\int_{\mathbb{T}^N}\int_{\mathbb{R}}\bar{f}^{\pm}_1(x,t,\xi)f^{\pm}_2(x,t,\xi)d\xi dx\Big|=0, \ a.s..
\end{eqnarray}
By a similar argument, passing to the limit $\delta\rightarrow 0$, it follows from  (\ref{e-23})-(\ref{qq-2}) that
\begin{eqnarray}\notag
\lim_{n\rightarrow \infty}\mathcal{E}_t(\gamma_n,\delta_n)=0, \quad a.s..
\end{eqnarray}
Without confusion, from now on, we write
\begin{eqnarray}\label{qq-4}
\lim_{\gamma, \delta\rightarrow 0}\mathcal{E}_t(\gamma,\delta)=0, \quad a.s..
\end{eqnarray}
In particular, when $t=0$, it holds that
\begin{eqnarray}\label{qq-5}
\lim_{\gamma, \delta\rightarrow 0}\mathcal{E}_0(\gamma,\delta)=0.
\end{eqnarray}
Now, we will make some estimates for $\tilde{I}(t)$, $\tilde{J}(t)$, $\tilde{K}(t)$ and $\tilde{H}(t)$.
We start with $\tilde{I}(t)$. Set
\[
\Gamma(\xi,\zeta)=\int^{\infty}_{\zeta}\int^{\xi}_{-\infty}\Upsilon(\xi',\zeta')|\xi'-\zeta'|\psi_{\delta}(\xi'-\zeta')d\xi'd\zeta',
\]
where $\Upsilon(\xi, \zeta)$ is the function appeared in Hypothesis H. Integration by parts yields that
\begin{eqnarray*}
 &&\int^t_0\int_{(\mathbb{T}^N)^2}\int_{\mathbb{R}^2}f_1\bar{f}_2(a(\xi)-a(\zeta))\cdot \nabla_x \alpha d\xi d\zeta dxdyds\\
&\leq& \int^t_0\int_{(\mathbb{T}^N)^2}\int_{\mathbb{R}^2}f_1\bar{f}_2\Upsilon(\xi,\zeta)|\xi-\zeta|| \nabla_x \rho_{\gamma}(x-y)|\psi_{\delta}(\xi-\zeta) d\xi d\zeta dxdyds\\
&=& -\int^t_0\int_{(\mathbb{T}^N)^2}\int_{\mathbb{R}^2}f_1\bar{f}_2\frac{\partial^2 \Gamma(\xi,\zeta)}{\partial\xi\partial\zeta} |\nabla_x \rho_{\gamma}(x-y)| d\xi d\zeta dxdyds\\
&=&\int^t_0\int_{(\mathbb{T}^N)^2}\int_{\mathbb{R}^2}\Gamma(\xi,\zeta)d\nu^{1}_{x,s}\otimes \nu^{2}_{y,s}(\xi,\zeta)|\nabla_x \rho_{\gamma}(x-y)|dxdyds\\
&\leq& C(q_0)\delta\int^t_0\int_{(\mathbb{T}^N)^2}\int_{\mathbb{R}^2}(1+|\xi|^{q_0}+|\zeta|^{q_0})d\nu^{1}_{x,s}\otimes \nu^{2}_{y,s}(\xi,\zeta)|\nabla_x \rho_{\gamma}(x-y)|dxdyds\\
&\leq& C(q_0)\delta\gamma^{-1}\int^t_0\int_{(\mathbb{T}^N)^2}\int_{\mathbb{R}^2}(1+|\xi|^{q_0}+|\zeta|^{q_0})d\nu^{1}_{x,s}\otimes \nu^{2}_{y,s}(\xi,\zeta)dxdyds,
\end{eqnarray*}
where we have used the fact that  $a(\cdot)$ is of polynomial growth with degree $q_0$ and (30) in \cite{D-V-1}.
Namely, we have obtained that
\begin{eqnarray*}
&&\Big|\int^t_0\int_{(\mathbb{T}^N)^2}\int_{\mathbb{R}^2}f_1\bar{f}_2(a(\xi)-a(\zeta))\cdot \nabla_x \alpha d\xi d\zeta dxdyds\Big|\\
&\leq& C(q_0)\delta\gamma^{-1}t+C(q_0)\delta\gamma^{-1}t\Big(\underset{0\leq s\leq t}{{\rm{ess\sup}}}\ \|u^{h}(s)\|^{q_0}_{L^{q_0}(\mathbb{T}^N)}+\underset{0\leq s\leq t}{{\rm{ess\sup}}}\ \|u(s)\|^{q_0}_{L^{q_0}(\mathbb{T}^N)}\Big). \quad a.s..
\end{eqnarray*}
Similar calculations lead to
\begin{eqnarray*}
&&\Big|\int^t_0\int_{(\mathbb{T}^N)^2}\int_{\mathbb{R}^2}\bar{f}_1f_2(a(\xi)-a(\zeta))\cdot \nabla_x \alpha d\xi d\zeta dxdyds\Big|\\
&\leq& C(q_0)\delta\gamma^{-1}t+C(q_0)\delta\gamma^{-1}t\Big(\underset{0\leq s\leq t}{{\rm{ess\sup}}}\ \|u^{h}(s)\|^{q_0}_{L^{q_0}(\mathbb{T}^N)}+\underset{0\leq s\leq t}{{\rm{ess\sup}}}\ \|u(s)\|^{q_0}_{L^{q_0}(\mathbb{T}^N)}\Big). \quad a.s..
\end{eqnarray*}
Combining the above inequalities, we get
\begin{eqnarray}\notag
|\tilde{I}(t)|&\leq& C(q_0)\delta\gamma^{-1}t+C(q_0)\delta\gamma^{-1}t\Big(\underset{0\leq s\leq t}{{\rm{ess\sup}}}\ \|u^{h}(s)\|^{q_0}_{L^{q_0}(\mathbb{T}^N)}+\underset{0\leq s\leq t}{{\rm{ess\sup}}}\ \|u(s)\|^{q_0}_{L^{q_0}(\mathbb{T}^N)}\Big).\quad a.s..
\end{eqnarray}
By (\ref{equ-29}) in Hypothesis H, we see that
\begin{eqnarray*}
\tilde{J}(t)&=&\int^t_0\int_{(\mathbb{T}^N)^2}\int_{\mathbb{R}^2}\alpha \sum_{k\geq 1}|g_k(x,\xi)-g_k(y,\zeta)|^2d\nu^{1}_{x,s}\otimes \nu^{2}_{y,s}(\xi,\zeta)dxdyds\\
&\leq&  D_1\int^t_0\int_{(\mathbb{T}^N)^2}\rho_{\gamma}(x-y)|x-y|^2\int_{\mathbb{R}^2}\psi_{\delta}(\xi-\zeta)d\nu^{1}_{x,s}\otimes \nu^{2}_{y,s}(\xi,\zeta)dxdyds\\
&& + D_1\int^t_0\int_{(\mathbb{T}^N)^2}\int_{\mathbb{R}^2}\rho_{\gamma}(x-y)\psi_{\delta}(\xi-\zeta)|\xi-\zeta|^2d\nu^{1}_{x,s}\otimes \nu^{2}_{y,s}(\xi,\zeta)dxdyds\\
&=:& \tilde{J}_{1}(t)+\tilde{J}_{2}(t).
\end{eqnarray*}
Noting that
\begin{eqnarray*}
\int_{\mathbb{R}^2}\psi_{\delta}(\xi,\zeta)d\nu^{1}_{x,s}\otimes \nu^{2}_{y,s}(\xi,\zeta)&\leq& \delta^{-1}, \quad a.s.,
\\
\int_{(\mathbb{T}^N)^2}\rho_{\gamma}(x-y)|x-y|^2dxdy&\leq&\gamma^2,
\end{eqnarray*}
we have
\begin{eqnarray}\label{e-12}
\tilde{J}_{1}(t)\leq  D_1\delta^{-1}\gamma^2t. \quad a.s..
\end{eqnarray}
For the term $\tilde{J}_{2}$, we have
\begin{eqnarray} \notag
\tilde{J}_{2}&\leq&  \delta D_1 \int^t_0\int_{(\mathbb{T}^N)^2}\int_{|\xi-\zeta|\leq \delta}\rho_{\gamma}(x-y)\psi_{\delta}(\xi-\zeta)|\xi-\zeta|d\nu^{1}_{x,s}\otimes \nu^{2}_{y,s}(\xi,\zeta) dxdyds\\
\label{e-11}
&\leq& \delta D_1C_{\psi}t, \quad a.s.,
\end{eqnarray}
where $C_{\psi}:=\sup_{\xi\in \mathbb{R}}\|\psi(\xi)\|$.
(\ref{e-12}) and (\ref{e-11}) together yield
\begin{eqnarray*}
\tilde{J}(t)\leq  D_1\delta^{-1}\gamma^2t+ D_1C_{\psi}\delta t, \quad a.s..
\end{eqnarray*}
By H\"{o}lder inequality and (\ref{equ-28}), we get
\begin{eqnarray*}
\tilde{H}(t)&\leq&
2\int^t_0|h(s)|_{l^2}\int_{(\mathbb{T}^N)^2}\int_{\mathbb{R}^2}\Big(\sum_{k\geq 1}|g_k(x,\xi)|^2\Big)^{\frac{1}{2}}\rho_{\gamma}(x-y)\chi(\xi,\zeta)d\nu^{1}_{x,s}\otimes \nu^{2}_{y,s}(\xi,\zeta)dxdyds\\
&\leq&
2D^{\frac{1}{2}}_0\int^t_0|h(s)|_{l^2}\int_{(\mathbb{T}^N)^2}\rho_{\gamma}(x-y)dxdyds\\
&\leq&
2D^{\frac{1}{2}}_0\int^t_0|h(s)|_{l^2}ds, \quad a.s.
\end{eqnarray*}
where we have used the fact that $\chi(\xi,\zeta)\leq 1$.

Combining all the above estimates, we deduce that
\begin{eqnarray}\notag
&&\int_{(\mathbb{T}^N)^2}\int_{\mathbb{R}^2}\rho_{\gamma} (x-y)\psi_{\delta}(\xi-\zeta)(f^{\pm}_1(x,t,\xi)\bar{f}^{\pm}_2(y,t,\zeta)+\bar{f}^{\pm}_1(x,t,\xi)f^{\pm}_2(y,t,\zeta))d\xi d\zeta dxdy\\
\notag
&\leq & \int_{(\mathbb{T}^N)^2}\int_{\mathbb{R}^2}\rho_{\gamma} (x-y)\psi_{\delta}(\xi-\zeta)(f_{1,0}(x,\xi)\bar{f}_{2,0}(y,\zeta)+\bar{f}_{1,0}(x,\xi)f_{2,0}(y,\zeta))d\xi d\zeta dxdy\\
\notag
&& +D_1\delta^{-1}\gamma^2t+ D_1C_{\psi}\delta t+C(q_0)\delta\gamma^{-1}t+2D^{\frac{1}{2}}_0\int^t_0|h(s)|_{l^2}ds\\
\label{e-20}
&&+C(q_0)\delta\gamma^{-1}t\Big(\underset{0\leq s\leq t}{{\rm{ess\sup}}}\ \|u^{h}(s)\|^{q_0}_{L^{q_0}(\mathbb{T}^N)}+\underset{0\leq s\leq t}{{\rm{ess\sup}}}\ \|u(s)\|^{q_0}_{L^{q_0}(\mathbb{T}^N)}\Big)+\tilde{K}(t), \quad a.s..
\end{eqnarray}
For $s\in (0,T)$, set
\[
R(s):=\int_{(\mathbb{T}^N)^2}\int_{\mathbb{R}^2}\rho_{\gamma} (x-y)\psi_{\delta}(\xi-\zeta)(f^{\pm}_1(x,s,\xi)\bar{f}^{\pm}_2(y,s,\zeta)+\bar{f}^{\pm}_1(x,s,\xi)f^{\pm}_2(y,s,\zeta))d\xi d\zeta dxdy.
\]
Then, we deduce from (\ref{e-20}) that
\begin{eqnarray*}
\underset{0\leq s\leq t}{{\rm{ess\sup}}}\ R(s)&\leq&  \int_{\mathbb{T}^N}\int_{\mathbb{R}}(f_{1,0}\bar{f}_{2,0}+\bar{f}_{1,0}f_{2,0})d\xi dx+\mathcal{E}_0(\gamma,\delta)\\
&&+D_1\delta^{-1}\gamma^2t+ D_1C_{\psi}\delta t+C(q_0)\delta\gamma^{-1}t+2D^{\frac{1}{2}}_0\int^t_0|h(s)|_{l^2}ds\\
&&+C(q_0)\delta\gamma^{-1}t\Big(\underset{0\leq s\leq t}{{\rm{ess\sup}}}\ \|u^{h}(s)\|^{q_0}_{L^{q_0}(\mathbb{T}^N)}+\underset{0\leq s\leq t}{{\rm{ess\sup}}}\ \|u(s)\|^{q_0}_{L^{q_0}(\mathbb{T}^N)}\Big)\\
&& +\sup_{0\leq s\leq t}|\tilde{K}|(s), \quad a.s.,
\end{eqnarray*}
where $\lim_{\gamma,\delta\rightarrow 0}\mathcal{E}_0(\gamma,\delta)=0$.

Taking the $L^2(\Omega)$-norm on both sides and using H\"{o
}lder inequality, we get that
\begin{eqnarray}\notag
\left(\mathbb{E}\Big|\underset{0\leq s\leq t}{{\rm{ess\sup}}}\ R(s)\Big|^2\right)^{\frac{1}{2}}
&\lesssim & \int_{\mathbb{T}^N}\int_{\mathbb{R}}(f_{1,0}\bar{f}_{2,0}+\bar{f}_{1,0}f_{2,0})d\xi dx+\mathcal{E}_0(\gamma,\delta)\\
\notag
&& +D_1\delta^{-1}\gamma^2 t+ D_1C_{\psi}\delta t+C(q_0)\delta\gamma^{-1}t+2t^{\frac{1}{2}}D^{\frac{1}{2}}_0\left(\mathbb{E}\int^t_0|h(s)|^2_{l^2}ds\right)^{\frac{1}{2}}\\
\label{e-1}
&& +C(q_0)\delta\gamma^{-1}t\mathcal{R} +\left(\mathbb{E}\Big|\sup_{s\in [0,t]}|\tilde{K}(s)|\Big|^2\right)^{\frac{1}{2}},
\end{eqnarray}
where
\begin{eqnarray*}
\mathcal{R}:=\left\{\Big(\mathbb{E}\underset{0\leq s\leq T}{{\rm{ess\sup}}}\ \|u^{h}(s)\|^{2q_0}_{L^{2q_0}(\mathbb{T}^N)}\Big)^{\frac{1}{2}}
+\Big(\mathbb{E}\underset{0\leq s\leq T}{{\rm{ess\sup}}}\ \|u(s)\|^{2q_0}_{L^{2q_0}(\mathbb{T}^N)}\Big)^{\frac{1}{2}}\right\}.
\end{eqnarray*}
In view of  (\ref{qq-8}) and (\ref{equation-1}), we have
\begin{eqnarray}\label{qq-29-1}
\mathcal{R}<+\infty.
\end{eqnarray}
To estimate the stochastic integral term, we use the Burkholder inequality to get
\begin{eqnarray}\label{e-7}
&&\mathbb{E}\Big|\sup_{s\in [0,t]}|\tilde{K}|(s)\Big|^2\\ \notag
&=& \mathbb{E}\Big|\sup_{s\in [0,t]}\sum_{k\geq 1}\int^s_0\int_{(\mathbb{T}^N)^2}\int_{\mathbb{R}^2}\chi(\xi,\zeta) \rho_{\gamma}(x-y)(g_k(x,\xi)-g_{k}(y,\zeta)) d \nu^{1}_{x,r}\otimes \nu^{2}_{y,r}(\xi,\zeta)dxdyd\beta_k(r)\Big|^2\\
\notag
&\lesssim& \mathbb{E}\Big[\int^t_0\sum_{k\geq 1}\Big|\int_{(\mathbb{T}^N)^2}\int_{\mathbb{R}^2}|g_k(x,\xi)-g_k(y,\zeta)|\rho_{\gamma}(x-y)\chi(\xi,\zeta) d \nu^{1}_{x,r}\otimes \nu^{2}_{y,r}(\xi,\zeta)dxdy\Big|^2dr\Big].
\end{eqnarray}
Recalling (\ref{e-6}) in Hypothesis H
\[
|g_k(x,\xi)-g_k(y,\zeta)|\leq C^1_k(|x-y|+|\xi-\zeta|),\quad \sum_{k\geq 1}|C^1_k|^2\leq \frac{D_1}{2},
\]
it follows from (\ref{e-7}) that
\begin{eqnarray*}
&&\mathbb{E}\Big|\sup_{s\in [0,t]}|\tilde{K}|(s)\Big|^2\\
&\lesssim& D_1\mathbb{E}\Big[\int^t_0\Big|\int_{(\mathbb{T}^N)^2}\int_{\mathbb{R}^2}(|x-y|+|\xi-\zeta|)\rho_{\gamma}(x-y)\chi(\xi,\zeta) d\nu^{1}_{x,r}\otimes \nu^{2}_{y,r}(\xi,\zeta)dxdy\Big|^2dr\Big].
\end{eqnarray*}
Since
\begin{eqnarray}\notag
\int_{(\mathbb{T}^N)^2}\int_{\mathbb{R}^2}|x-y|\rho_{\gamma}(x-y) \chi(\xi,\zeta)d \nu^{1}_{x,r}\otimes \nu^{2}_{y,r}(\xi,\zeta)dxdy\leq \gamma, \quad a.s..
\end{eqnarray}
it follows that
\begin{eqnarray}\label{e-24}
\mathbb{E}\Big|\sup_{s\in [0,t]}|\tilde{K}|(s)\Big|^2
\lesssim D_1\mathbb{E}\Big[\int^t_0\Big|\gamma+\int_{(\mathbb{T}^N)^2}|u^{h,\pm}-u^{\pm}|\rho_{\gamma}(x-y) dxdy\Big|^2dr\Big].
\end{eqnarray}
With the help of (\ref{111.1-2}), (\ref{e-23}) and (\ref{e-22}), we deduce that
\begin{eqnarray}\notag
&& \int_{(\mathbb{T}^N)^2}|u^{h,\pm}(x,r)-u^{\pm}(y,r)|\rho_{\gamma}(x-y) dxdy\\ \notag
&=& \int_{(\mathbb{T}^N)^2}\Big((u^{h,\pm}(x,r)-u^{\pm}(y,r))^{+}+(u^{h,\pm}(x,r)-u^{\pm}(y,r))^{-}\Big)\rho_{\gamma}(x-y) dxdy\\
\notag
&=&\int_{(\mathbb{T}^N)^2}\int_{\mathbb{R}}(\bar{f}^{\pm}_1(x,r,\xi)f^{\pm}_2(y,r,\xi)+f^{\pm}_1(x,r,\xi)\bar{f}^{\pm}_2(y,r,\xi))\rho_{\gamma}(x-y)d \xi  dxdy\\ \notag
&\leq& 4\delta+\int_{(\mathbb{T}^N)^2}\int_{\mathbb{R}^2}(\bar{f}^{\pm}_1(x,r,\xi)f^{\pm}_2(y,r,\zeta)+f^{\pm}_1(x,r,\xi)\bar{f}^{\pm}_2(y,r,\zeta))\rho_{\gamma}(x-y)\psi_{\delta}(\xi-\zeta) d \xi d\zeta dxdy\\
\label{e-25}
&=& 4\delta+R(r). \quad a.s..
\end{eqnarray}
Combining (\ref{e-24}) and (\ref{e-25}), we obtain  that
\begin{eqnarray}\notag
\mathbb{E}\Big|\sup_{s\in [0,t]}|\tilde{K}(s)|\Big|^2
&\lesssim& D_1\mathbb{E}\Big[\int^t_0\Big|\gamma+4\delta+R(r)\Big|^2dr\Big]\\
\notag
&\leq & 2D_1\mathbb{E}\Big[\int^t_0\Big|\gamma+4\delta\Big|^2dr+\int^t_0|R(r)|^2dr\Big]
\\
\label{e-2}
&\leq &
2D_1t |\gamma+4\delta|^2+2D_1\mathbb{E}\Big(\int^t_0R^2(r)dr\Big).
\end{eqnarray}
Substitute  (\ref{e-2}) back into (\ref{e-1}) to get
\begin{eqnarray*}\notag
&&\left(\mathbb{E}\Big|\underset{0\leq s\leq t}{{\rm{ess\sup}}}\ R(s)\Big|^2\right)^{\frac{1}{2}}\\
&\lesssim & \int_{\mathbb{T}^N}\int_{\mathbb{R}}(f_{1,0}\bar{f}_{2,0}+\bar{f}_{1,0}f_{2,0})d\xi dx+\mathcal{E}_0(\gamma,\delta)+D_1\delta^{-1}\gamma^2t+ D_1C_{\psi}\delta t+C(q_0)\delta\gamma^{-1}t\\
&&+2t^{\frac{1}{2}}D^{\frac{1}{2}}_0\left(\mathbb{E}\int^t_0|h(s)|^2_{l^2}ds\right)^{\frac{1}{2}}
+C(q_0)\delta\gamma^{-1}\mathcal{R}t +2^{\frac{1}{2}}D^{\frac{1}{2}}_1t^{\frac{1}{2}}
|\gamma+4\delta| +2^{\frac{1}{2}}D^{\frac{1}{2}}_1\Big(\mathbb{E}\Big(\int^t_0R^2(r)dr\Big)\Big)^{\frac{1}{2}}.
\end{eqnarray*}
Squaring the above inequality, we get
\begin{eqnarray}\notag
&&\mathbb{E}\Big|\underset{0\leq s\leq t}{{\rm{ess\sup}}}\ R(s)\Big|^2\\ \notag
&\lesssim & \Big[\int_{\mathbb{T}^N}\int_{\mathbb{R}}(f_{1,0}\bar{f}_{2,0}+\bar{f}_{1,0}f_{2,0})d\xi dx+\mathcal{E}_0(\gamma,\delta)+D_1\delta^{-1}\gamma^2t+ D_1C_{\psi}\delta t+C(q_0)\delta\gamma^{-1}t\\ \notag
&&+2t^{\frac{1}{2}}D^{\frac{1}{2}}_0\left(\mathbb{E}\int^t_0|h(s)|^2_{l^2}ds\right)^{\frac{1}{2}}
+C(q_0)\delta\gamma^{-1}\mathcal{R}t +2^{\frac{1}{2}}D^{\frac{1}{2}}_1t^{\frac{1}{2}}
|\gamma+4\delta|\Big]^2\\
\label{e-8}
&& +2D_1\int^t_0\Big(\mathbb{E}\Big|\underset{0\leq s\leq r}{{\rm{ess\sup}}}\ R(s)\Big|^2\Big)dr.
\end{eqnarray}
Applying Gronwall inequality to (\ref{e-8}), we get
\begin{eqnarray}\notag
&&\left(\mathbb{E}\Big|\underset{0\leq s\leq T}{{\rm{ess\sup}}}\ R(s)\Big|^2\right)^{\frac{1}{2}}\\ \notag
&\lesssim& e^{ D_1T}\Big[\int_{\mathbb{T}^N}\int_{\mathbb{R}}(f_{1,0}\bar{f}_{2,0}+\bar{f}_{1,0}f_{2,0})d\xi dx+\mathcal{E}_0(\gamma,\delta)+D_1\delta^{-1}\gamma^2T+ D_1C_{\psi}\delta T+C(q_0)\delta\gamma^{-1}T
\\
\label{qq-10}
&&  \quad \quad +2T^{\frac{1}{2}}D^{\frac{1}{2}}_0\Big(\mathbb{E}\int^T_0|h(s)|^2_{l^2}ds\Big)^{\frac{1}{2}}
+C(q_0)\delta\gamma^{-1}\mathcal{R}T +2^{\frac{1}{2}}D^{\frac{1}{2}}_1T^{\frac{1}{2}}
|\gamma+4\delta|\Big].
\end{eqnarray}
Let
\[
Q(s):=\int_{(\mathbb{T}^N)^2}\int_{(\mathbb{R})^2}\rho_{\gamma}(x-y)\psi_{\gamma}(\xi-\zeta)(f^{\pm}_2(s,x,\xi)\bar{f}^{\pm}_2(s,y,\zeta)+\bar{f}^{\pm}_2(s,x,\xi){f}^{\pm}_2(s,y,\zeta))d\xi d\zeta dxdy.
\]
Applying the same arguments to $f^{\pm}_2$ and $\bar{f}^{\pm}_2$ (in this case, $h=0$ ), we can show that
\begin{eqnarray}\notag
\left(\mathbb{E}\Big|\underset{0\leq s\leq T}{{\rm{ess\sup}}}\ Q(s)\Big|^2\right)^{\frac{1}{2}}
&\lesssim & e^{D_1T}\Big[\mathcal{E}_0(\gamma,\delta)+D_1\delta^{-1}\gamma^2T+ D_1C_{\psi}\delta T+C(q_0)\delta\gamma^{-1}T\\
\label{qq-12}
&&
+C(q_0)\delta\gamma^{-1}\mathcal{R}T +2^{\frac{1}{2}}D^{\frac{1}{2}}_1T^{\frac{1}{2}}
|\gamma+4\delta|\Big].
\end{eqnarray}
On the other hand, from (\ref{qq-3}), it follows that
\begin{eqnarray}\notag
&&\left(\mathbb{E} \Big|\underset{0\leq s\leq T}{{\rm{ess\sup}}}\ \int_{\mathbb{T}^N}\int_{\mathbb{R}}(f^{\pm}_1(s,x,\xi)\bar{f}^{\pm}_2(s,x,\xi)+\bar{f}^{\pm}_1(s,x,\xi){f}^{\pm}_2(s,x,\xi))d\xi dx\Big|^2\right)^{\frac{1}{2}}\\ \label{eee-1}
&\lesssim& \left(\mathbb{E}\Big|\underset{0\leq s\leq T}{{\rm{ess\sup}}}\ |\mathcal{E}_s(\gamma,\delta)|\Big|^2\right)^{\frac{1}{2}}+\left(\mathbb{E}\Big|\underset{0\leq s\leq T}{{\rm{ess\sup}}}\ R(s)\Big|^2\right)^{\frac{1}{2}}.
\end{eqnarray}
Now  we will provide estimates for $ \left(\mathbb{E}\Big|\underset{0\leq s\leq T}{{\rm{ess\sup}}}\ |\mathcal{E}_s(\gamma,\delta)|\Big|^2\right)^{\frac{1}{2}}$.
For any $s\in (0,T)$, we write
\begin{eqnarray*}
\mathcal{E}_s(\gamma, \delta)&=&\int_{(\mathbb{T}^N)^2}\int_{\mathbb{R}^2}(f^{\pm}_1(x,s,\xi)\bar{f}^{\pm}_2(y,s,\zeta)+\bar{f}^{\pm}_1(x,s,\xi){f}^{\pm}_2(y,s,\zeta))\rho_{\gamma}(x-y)\psi_{\delta}(\xi-\zeta)dxdyd\xi d\zeta
\\
&&-\int_{\mathbb{T}^N}\int_{\mathbb{R}}(f^{\pm}_1(x,s,\xi)\bar{f}^{\pm}_2(x,s,\xi)+\bar{f}^{\pm}_1(x,s,\xi){f}^{\pm}_2(x,s,\xi))d\xi dx\\
&=& \Big[\int_{(\mathbb{T}^N)^2}\int_{\mathbb{R}}\rho_{\gamma}(x-y)(f^{\pm}_1(x,s,\xi)\bar{f}^{\pm}_2(y,s,\xi)+\bar{f}^{\pm}_1(x,s,\xi){f}^{\pm}_2(y,s,\xi))d\xi dxdy\\
&& -\int_{\mathbb{T}^N}\int_{\mathbb{R}}(f^{\pm}_1(x,s,\xi)\bar{f}^{\pm}_2(x,s,\xi)+\bar{f}^{\pm}_1(x,s,\xi){f}^{\pm}_2(x,s,\xi))d\xi dx\Big]\\
&& +\Big[\int_{(\mathbb{T}^N)^2}\int_{\mathbb{R}^2}(f^{\pm}_1(x,s,\xi)\bar{f}^{\pm}_2(y,s,\zeta)+\bar{f}^{\pm}_1(x,s,\xi){f}^{\pm}_2(y,s,\zeta))\rho_{\gamma}(x-y)\psi_{\delta}(\xi-\zeta)dxdyd\xi d\zeta\\
&& -\int_{(\mathbb{T}^N)^2}\int_{\mathbb{R}}\rho_{\gamma}(x-y)(f^{\pm}_1(x,s,\xi)\bar{f}^{\pm}_2(y,s,\xi)+\bar{f}^{\pm}_1(x,s,\xi){f}^{\pm}_2(y,s,\xi))d\xi dxdy \Big]\\
&=:&H_1+H_2.
\end{eqnarray*}
By (\ref{e-23}) and (\ref{e-22}), we have
\begin{eqnarray}\label{qq-15}
|H_2|\leq 4\delta, \quad a.s..
\end{eqnarray}
On the other hand,
\begin{eqnarray*}
|H_1|&\leq& \Big|\int_{(\mathbb{T}^N)^2}\rho_{\gamma}(x-y)\int_{\mathbb{R}}I_{u^{h,\pm}(x,s)>\xi}(I_{u^{\pm}(x,s)\leq \xi}-I_{u^{\pm}(y,s)\leq \xi})d\xi dxdy\Big|\\
&& +\Big|\int_{(\mathbb{T}^N)^2}\rho_{\gamma}(x-y)\int_{\mathbb{R}}I_{u^{h,\pm}(x,s)\leq\xi}(I_{u^{\pm}(x,s)> \xi}-I_{u^{\pm}(y,s)> \xi})d\xi dxdy\Big|\\
&\leq& 2\int_{(\mathbb{T}^N)^2}\rho_{\gamma}(x-y)|u^{\pm}(x,s)-u^{\pm}(y,s)|dxdy, \quad a.s..
\end{eqnarray*}
Using (\ref{e-23}) and (\ref{e-22}) again, it follows that
\begin{eqnarray*}
&&\int_{(\mathbb{T}^N)^2}\rho_{\gamma}(x-y)|u^{\pm}(x,s)-u^{\pm}(y,s)|dxdy\\
&=& \int_{(\mathbb{T}^N)^2}\int_{\mathbb{R}}\rho_{\gamma}(x-y)(f^{\pm}_2(x,s,\xi)\bar{f}^{\pm}_2(y,s,\xi)+\bar{f}^{\pm}_2(x,s,\xi){f}^{\pm}_2(y,s,\xi))d\xi dxdy\\
&\leq& \int_{(\mathbb{T}^N)^2}\int_{\mathbb{R}^2}\rho_{\gamma}(x-y)\psi_{\delta}(\xi-\zeta)(f^{\pm}_2(x,s,\xi)\bar{f}^{\pm}_2(y,s,\zeta)+\bar{f}^{\pm}_2(x,s,\xi){f}^{\pm}_2(y,s,\zeta))d\xi d\zeta dxdy+4\delta\\
&=&Q(s)+4\delta, \quad a.s..
\end{eqnarray*}
Hence,
\begin{eqnarray}\label{qq-14}
|H_1|\leq 2Q(s)+8\delta, \quad a.s..
\end{eqnarray}
Collecting (\ref{qq-15}) and (\ref{qq-14}) yields
\begin{eqnarray*}
|\mathcal{E}_s(\gamma, \delta)|\leq 2Q(s)+12\delta, \quad a.s.,
\end{eqnarray*}
By (\ref{qq-12}), we deduce that
\begin{eqnarray}\notag
&&\Big(\mathbb{E}\big|\underset{0\leq s\leq T}{{\rm{ess\sup}}}\ |\mathcal{E}_s(\gamma,\delta)|\big|^2\Big)^{\frac{1}{2}}\\ \notag
&\lesssim& \Big(\mathbb{E}|\underset{0\leq s\leq T}{{\rm{ess\sup}}}\ Q(s)|^2\Big)^{\frac{1}{2}}+\delta\\
\notag
&\lesssim&e^{D_1T}\Big[\mathcal{E}_0(\gamma,\delta)+D_1\delta^{-1}\gamma^2T+ D_1C_{\psi}\delta T+C(q_0)\delta\gamma^{-1}T\\
\label{qq-16}
&&
+C(q_0)\delta\gamma^{-1}\mathcal{R}T +2^{\frac{1}{2}}D^{\frac{1}{2}}_1T^{\frac{1}{2}}
|\gamma+4\delta|\Big]+\delta.
\end{eqnarray}
Combining (\ref{qq-10}) and (\ref{qq-16}), we deduce from (\ref{eee-1}) that
\begin{eqnarray*}\notag
&&\left(\mathbb{E} \Big|\underset{0\leq s\leq T}{{\rm{ess\sup}}}\ \int_{\mathbb{T}^N}\int_{\mathbb{R}}(f^{\pm}_1(s,x,\xi)\bar{f}^{\pm}_2(s,x,\xi)+\bar{f}^{\pm}_1(s,x,\xi){f}^{\pm}_2(s,x,\xi))d\xi dx\Big|^2\right)^{\frac{1}{2}}\\ \notag
&\lesssim&e^{ D_1T}\Big[\int_{\mathbb{T}^N}\int_{\mathbb{R}}(f_{1,0}\bar{f}_{2,0}+\bar{f}_{1,0}f_{2,0})d\xi dx+2\mathcal{E}_0(\gamma,\delta)+2D_1\delta^{-1}\gamma^2T+2D_1C_{\psi}\delta T+2C(q_0)\delta\gamma^{-1}T
\\
&& +2T^{\frac{1}{2}}D^{\frac{1}{2}}_0\Big(\mathbb{E}\int^T_0|h(s)|^2_{l^2}ds\Big)^{\frac{1}{2}}
+2C(q_0)\delta\gamma^{-1}\mathcal{R}T +3D^{\frac{1}{2}}_1T^{\frac{1}{2}}
|\gamma+4\delta|\Big]+\delta.
\end{eqnarray*}
Note that we have $f^{\pm}_1(x,s,\xi)=I_{u^{h,\pm}(s,x)>\xi}$ and $f^{\pm}_2(x,s,\xi)=I_{u^{\pm}(s,x)>\xi}$ with initial data $f_{1,0}=I_{\eta>\xi}$ and ${f}_{2,0}=I_{\eta>\xi}$, respectively. In view of (\ref{111.1-2}), we can rewrite the above inequality as
\begin{eqnarray}\label{e-26}
\left(\mathbb{E}\Big|\underset{0\leq s\leq T}{{\rm{ess\sup}}}\ \|u^{h,\pm}(s)-u^{\pm}(s)\|_{L^1(\mathbb{T}^N)}\Big|^2\right)^{\frac{1}{2}}
\lesssim r(\gamma, \delta),
\end{eqnarray}
where
\begin{eqnarray}\notag
&&r(\gamma, \delta)\\ \notag
&:=&e^{D_1T}\Big[2\mathcal{E}_0(\gamma,\delta)+2D_1\delta^{-1}\gamma^2T+2D_1C_{\psi}\delta T+2C(q_0)\delta\gamma^{-1}T
\\
\label{e-10}
&& +2T^{\frac{1}{2}}D^{\frac{1}{2}}_0\Big(\mathbb{E}\int^T_0|h(s)|^2_{l^2}ds\Big)^{\frac{1}{2}}
+2C(q_0)\delta\gamma^{-1}\mathcal{R}T +3D^{\frac{1}{2}}_1T^{\frac{1}{2}}
|\gamma+4\delta|\Big]+\delta.
\end{eqnarray}
Taking
\[
\delta=\gamma^{\frac{4}{3}},
\]
we have
\begin{eqnarray*}
&&r(\gamma, \delta)\\ \notag
&:=&e^{D_1T}\Big[2\mathcal{E}_0(\gamma,\delta)+2D_1\gamma^{\frac{2}{3}}T+ 2D_1C_{\psi}\gamma^{\frac{4}{3}} T+2C(q_0)\gamma^{\frac{1}{3}}T
\\
&& +2T^{\frac{1}{2}}D^{\frac{1}{2}}_0\Big(\mathbb{E}\int^T_0|h(s)|^2_{l^2}ds\Big)^{\frac{1}{2}}
+2C(q_0)\gamma^{\frac{1}{3}}\mathcal{R}T +3D^{\frac{1}{2}}_1T^{\frac{1}{2}}
|\gamma+4\gamma^{\frac{4}{3}}|\Big]+\gamma^{\frac{4}{3}}.
\end{eqnarray*}
Let $\gamma\rightarrow 0$ to get
\begin{eqnarray*}
 \lim_{\gamma\rightarrow 0} r(\gamma, \delta)\leq 2T^{\frac{1}{2}}D^{\frac{1}{2}}_0e^{D_1T}\Big(\mathbb{E}\int^T_0|h(s)|^2_{l^2}ds\Big)^{\frac{1}{2}}.
\end{eqnarray*}
Therefore, we deduce from (\ref{e-26}) that
\begin{eqnarray*}
\left(\mathbb{E}\Big|\underset{0\leq s\leq T}{{\rm{ess\sup}}}\ \|u^{h,\pm}(s)-u^{\pm}(s)\|_{L^1(\mathbb{T}^N)}\Big|^2\right)^{\frac{1}{2}}
\leq2\mathcal{D}T^{\frac{1}{2}}D^{\frac{1}{2}}_0e^{D_1T}\Big(\mathbb{E}\int^T_0|h(s)|^2_{l^2}ds\Big)^{\frac{1}{2}},
\end{eqnarray*}
which implies
\begin{eqnarray}
\mathbb{E}\Big|\int_0^T\|u^{h}(t)-u(t)\|_{L^1(\mathbb{T}^N)}dt\Big|^2\leq 4\mathcal{D}^2D_{0}T^3e^{2D_1T}\mathbb{E}\int^T_0|h(s)|^2_{l^2}ds.
\end{eqnarray}
We complete the proof.

\end{proof}

\noindent{\bf  Acknowledgements}\quad  This work is partly supported by National Natural Science Foundation of China (No. 11671372, 11971456, 11721101, 11801032, 11971227), Key Laboratory of Random Complex Structures and Data Science, Academy of Mathematics and Systems Science, Chinese Academy of Sciences (No. 2008DP173182), Beijing Institute of Technology Research Fund Program for Young Scholars.

\def\refname{ References}


\begin{thebibliography}{2}
\bibitem[AWC]{K-P-J}
K. Ammar, P. Wittbolda and J. Carrillo. \emph{Scalar conservation laws with general boundary condition and continuous flux function}. J. Differential Equations 228, no. 1, 111-139 (2006).

\bibitem[BH]{BH2}
B. Boufoussi and S. Hajji. \emph{Transportation inequalities for stochastic heat equations.} Statist. Probab. Lett. 139, 75-83 (2018).
\bibitem[Da]{Dafermos} C.M. Dafermos. \emph{Hyperbolic Conservation Laws in Continuum Physics}. 2nd edn. Berlin, Springer (2005).


\bibitem[DHV]{DHV} A. Debussche, M. Hofmanov\'{a} and J. Vovelle. \emph{Degenerate parabolic stochastic partial differential equations: Quasilinear case}. Ann. Probab. 44, no. 3, 1916-1955 (2016).

\bibitem[DV-1]{D-V-1} A. Debussche and J. Vovelle.
\emph{Scalar conservation laws with stochastic forcing (revised version). http://math.univ-lyon1.fr/vovelle/DebusscheVovelleRevised}.
J. Funct. Anal. 259, no. 4, 1014-1042 (2010).
\bibitem[DV-2]{D-V-2} A. Debussche and J. Vovelle.
\emph{Invariant measure of scalar first-order conservation laws with stochastic forcing}. Probab. Theory Related Fields 163, no. 3-4, 575-611 (2015).

\bibitem[DGW]{DGW}
H. Djellout, A. Guillin and L. Wu. \emph{
Transportation cost-information inequalities and applications to random dynamical systems and diffusions.} Ann. Probab. 32, no. 3B, 2702-2732 (2004).

\bibitem[DWZZ]{DWZZ} Z. Dong, J.-L. Wu, R. Zhang and  T. Zhang. \emph{Large deviation principles for first-order scalar conservation laws with stochastic forcing.}
Ann. Appl. Probab. 30, no. 1, 324-367 (2020).

\bibitem[DP]{DP}
D.P. Dubhashi and A. Panconesi. \emph{Concentration of measure for the analysis of randomized algorithms.} Cambridge University Press, 2009.

 \bibitem[FN]{F-N} J. Feng and D. Nualart. \emph{Stochastic scalar conservation laws.} J. Funct. Anal. 255, no. 2, 313-373 (2008).
\bibitem[GRS]{GRS}
N. Gozlan, C. Roberto and P.-M. Samson. \emph{Characterization of Talagrand transport-entropy inequalities in metric spaces.}
Ann. Probab. 41, no. 5, 3112-3139 (2013).

\bibitem[KS]{KS}
D.~Khoshnevisan and A. Sarantsev. \emph{Talagrand concentration inequalities for stochastic partial differential equations.} Stoch. Partial Differ. Equ. Anal. Comput. 7, no. 4, 679-698 (2019).

\bibitem[K]{K} J.U. Kim. \emph{On a stochastic scalar conservation law}. Indiana Univ. Math. J. 52 227-256 (2003).

\bibitem[Kr-1]{K-69} S.N.  Kru\v{z}kov. \emph{Generalized solutions of the Cauchy problem in the large for first order nonlinear equations.} Dokl. Akad. Nauk. SSSR 187 29-32 (1969).

\bibitem[Kr-2]{K-70} S.N.  Kru\v{z}kov. \emph{First order quasilinear equations with several independent variables.} Mat. Sb. (N.S.) 81 (123) 228-255 (1970).

\bibitem[La]{La}
D. Lacker. \emph{Liquidity, risk measures, and concentration of measure.} Math. Oper. Res. 43, no. 3, 813-837 (2018).
\bibitem[Le]{L}
M. Ledoux. \emph{The concentration of measure phenomenon.}
American Mathematical Society, 2001.

\bibitem[LPT]{L-P-T} P.L. Lions, B. Perthame and E. Tadmor. \emph{A kinetic formulation of multidimensional scalar conservation laws and related equations}. J. of A.M.S., 7, 169-191 (1994).
\bibitem[LW]{LW}
Y. Li and X. Wang. \emph{Transportation inequalities for stochastic wave equation.} arXiv:1811.06385., 2018.
\bibitem[M1]{M-1}
K. Marton. \emph{Bounding $\overline d$-distance by informational divergence: a method to prove measure concentration.} Ann. Probab. 24, no. 2, 857-866 (1996).
\bibitem[M2]{M-2}
K. Marton. \emph{A measure concentration inequality for contracting Markov chains.} Geom. Funct. Anal. 6 , no. 3, 556-571 (1996).
\bibitem[Ma]{M}
P. Massart. \emph{Concentration inequalities and model selection.} Lecture Notes in Mathematics, volume 1896, Springer, 2007.
\bibitem[OV]{OV}
F. Otto and C. Villani. \emph{Generalization of an inequality by Talagrand and links with the logarithmic Sobolev inequality.} J. Funct. Anal. 173, no. 2, 361-400 (2000).
\bibitem[P]{P}
S. Pal. \emph{Concentration for multidimensional diffusions and their boundary local times.} Probab. Theory Related Fields 154, no. 1-2, 225-254 (2012).
\bibitem[PS]{PS}
S. Pal and A. Sarantsev. \emph{A Note on Transportation Cost Inequalities for Diffusions with Reflections.}
Electron. Commun. Probab. 24, Paper No. 21, 11 pp (2019).
\bibitem[PS1]{PS1}
S. Pal and M. Shkolnikov. \emph{Concentration of measure for Brownian particle systems interacting through their ranks.} Ann. Appl. Probab. 24, no. 4, 1482-1508 (2014).
\bibitem[SZ]{SZ}
S. Shang and T. Zhang. \emph{Talagrand Concentration Inequalities for Stochastic Heat-Type Equations  under Uniform Distance.} Electron. J. Probab. 24, Paper No. 129, 15 pp (2019).
\bibitem[T1]{T1}
M. Talagrand. \emph{Concentration of measure and isoperimetric inequalities in product spaces.} Inst. Hautes \'{e}tudes Sci. Publ. Math. No. 81, 73-205 (1995).
\bibitem[T2]{T2}
M. Talagrand. \emph{Transportation cost for Gaussian and other product measures.} Geom. Funct. Anal. 6, no. 3, 587-600 (1996).
\bibitem[T3]{T3}
M. Talagrand. \emph{New concentration inequalities in product spaces.} Invent. Math. 126, no. 3, 505-563 (1996).
\bibitem[U]{U}
A. S. \"{U}st\"{u}nel. \emph{Transportation cost inequalities for diffusions under uniform distance.} In Stochastic analysis and related topics, pp. 203-214. Springer, 2012.

\bibitem[VW]{V-W} G. Vallet, P. Wittbold: \emph{On a stochastic first-order hyperbolic equation in a bounded domain}.
Infin. Dimens. Anal. Quantum Probab. Relat. Top. 12, no. 4, 613-651 (2009).

\end{thebibliography}
\end{document}